\documentclass[12pt]{amsart}
\usepackage{amssymb}
\usepackage{mathrsfs}
\usepackage{hyperref}
\usepackage{fullpage}

\newtheorem{theorem}{Theorem}
\numberwithin{theorem}{section}
\newtheorem{lemma}[theorem]{Lemma}

\newtheorem{definition}[theorem]{Definition}

\newcommand{\Aut}{\ensuremath{\mathrm{Aut}}}

\newcommand{\cl}{\ensuremath{\mathrm{cl}}}

\newcommand{\codim}{\ensuremath{\mathrm{codim}}}

\newcommand{\cond}{\ensuremath{\mathrm{cond}}}

\newcommand{\End}{\ensuremath{\mathrm{End}}}

\newcommand{\Fr}{\ensuremath{\mathrm{Fr}}}
\newcommand{\Frss}{\ensuremath{\mathrm{Fr\text{-}ss}}}

\newcommand{\Gal}{\ensuremath{\mathrm{Gal}}}

\newcommand{\gln}{\ensuremath{\operatorname{GL}}}

\newcommand{\intal}{\ensuremath{\mathrm{intal}}}

\newcommand{\Iw}{\ensuremath{\mathrm{Iw}}}

\newcommand{\loc}{\ensuremath{\mathrm{loc}}}

\newcommand{\modu}{\ensuremath{\mathrm{\;mod\;}}}

\newcommand{\ord}{\ensuremath{\mathrm{ord}}}

\newcommand{\Qbar}{\ensuremath{\overline{Q}}}

\newcommand{\res}{\ensuremath{\mathrm{res}}}

\newcommand{\Sp}{\ensuremath{\mathrm{Sp}}}

\newcommand{\tr}{\ensuremath{\mathrm{tr}}}

\newcommand{\ur}{\ensuremath{\mathrm{ur}}}

\newcommand{\WBC}{\ensuremath{\mathrm{WBC}}}
\newcommand{\WD}{\ensuremath{\mathrm{WD}}}

\newcommand{\cf}{{\it cf.}}

\newcommand{\ie}{{\it i.e.}}
\newcommand{\loccit}{{\it loc.\,cit}}

\newcommand{\hra}{\hookrightarrow}

\newcommand{\xra}{\xrightarrow}

\newcommand{\bbA}{\ensuremath{\mathbb{A}}}

\newcommand{\bbC}{\ensuremath{\mathbb{C}}}

\newcommand{\bbQ}{\ensuremath{\mathbb{Q}}}
\newcommand{\bbQbar}{\ensuremath{\overline\bbQ}}
\newcommand{\bbQp}{\ensuremath{{\mathbb{Q}_p}}}
\newcommand{\bbQpbar}{{\ensuremath{\overline{\mathbb{Q}}_p}}}
\newcommand{\bbQl}{\ensuremath{\mathbb{Q}_\ell}}
\newcommand{\bbQlbar}{\ensuremath{\overline\bbQ_\ell}}

\newcommand{\bbT}{\ensuremath{\mathbb{T}}}
\newcommand{\bbZ}{\ensuremath{\mathbb{Z}}}

\newcommand{\bbZp}{\ensuremath{{\mathbb{Z}_p}}}

\newcommand{\calF}{\ensuremath{\mathcal{F}}}

\newcommand{\calO}{\ensuremath{\mathcal{O}}}

\newcommand{\calR}{\ensuremath{\mathcal{R}}}

\newcommand{\fraka}{\ensuremath{\mathfrak{a}}}

\newcommand{\frakl}{\ensuremath{\mathfrak{l}}}
\newcommand{\frakm}{\ensuremath{\mathfrak{m}}}

\newcommand{\frakp}{\ensuremath{\mathfrak{p}}}

\newcommand{\scrL}{\ensuremath{\mathscr{L}}}
\newcommand{\scrLbar}{\ensuremath{{\overline{\mathscr{ L}} }  }}

\newcommand{\scrO}{\ensuremath{\mathscr{O}}}

\begin{document}
\title{Conductors in $p$-adic families}

\author{Jyoti Prakash Saha}

\address{Max Planck Institute for Mathematics, Vivatsgasse 7, 53111 Bonn, Germany}

\email{saha@mpim-bonn.mpg.de}




\begin{abstract}
Given a Weil-Deligne representation of the Weil group of an $\ell$-adic number field with coefficients in a domain $\mathscr{O}$, we show that its pure specializations have the same conductor. More generally, we prove that the conductors of a collection of pure representations are equal if they lift to Weil-Deligne representations over domains containing $\mathscr{O}$ and the traces of these lifts are parametrized by a pseudorepresentation over $\mathscr{O}$.  
\end{abstract}
\maketitle

\section{Introduction}
The aim of this article is to study the variation of automorphic and Galois conductors in $p$-adic families of automorphic Galois representations, for instance, in Hida families and eigenvarieties. We relate the variation of Galois conductors in families to purity of $p$-adic automorphic Galois representations at the finite places not dividing $p$ and for the variation of automorphic conductors, we use local-global compatibility. We establish the constancy of tame conductors at arithmetic points lying along irreducible components of $p$-adic families. 

\subsection{Motivation}
In \cite{HidaGalrepreord, HidaIwasawa}, Hida showed that the $p$-ordinary eigen cusp forms, \ie, the normalized eigen cusp forms whose $p$-th Fourier coefficients are $p$-adic units (with respect to fixed embeddings of $\bbQbar$ in $\bbC$ and $\bbQpbar$), can be put in $p$-adic families. More precisely, he showed that for each positive integer $N$ and an odd prime $p$ with $p\nmid N$ and $Np\geq 4$, there is a subset of the set of $\bbQpbar$-specializations of the universal $p$-ordinary Hecke algebra $h^\ord(N; \bbZp)$, called the set of arithmetic specializations, such that there is a one-to-one correspondence between the arithmetic specializations of $h^\ord(N;\bbZp)$ and the $p$-ordinary $p$-stabilized normalized eigen cusp forms of tame level a divisor of $N$. It turns out that the tame conductors of the Galois representations attached to the arithmetic specializations remain constant along irreducible components of $h^\ord(N; \bbZp)$ (see \cite[Theorem 2.3, 2.4]{HidaPadicmeasure}). Following Hida's construction of families of ordinary cusp forms, further examples of families of automorphic Galois representations are constructed, for instance, Hida families of ordinary automorphic representations of definite unitary groups, families of overconvergent forms (see the works of Hida \cite{HidaControThmPNearlyOrdinaryCohGroupSLn}, Coleman, Mazur \cite{ColemanMazurEigencurve}, Chenevier \cite{ChenevierGLn}, Bella\"iche, Chenevier \cite{BellaicheChenevierU3} {\it et.\,al.}). The aim of this article is to understand the variation of automorphic and Galois conductors in these families of automorphic Galois representations. Since the restrictions of $p$-adic automorphic Galois representations to decomposition groups at places outside $p$ are known to be pure in many cases, we focus on the variation of conductors in families of pure representations of local Galois groups. 

\subsection{Results obtained}
Let $p,\ell$ be two distinct primes and $K$ be a finite extension of $\bbQl$. Let $\scrO$ be an integral domain containing $\bbQ$. In theorem \ref{Thm: conductor big}, we show that given any Weil-Deligne representation of the Weil group $W_K$ of $K$ with coefficients in $\scrO$, its conductor coincides with the conductors of its pure specializations over $\bbQpbar$. Next, in theorem \ref{Thm: conductor pseudorepresentation}, we show that a collection of pure representations of $W_K$ over $\bbQpbar$ have the same conductor if they lift to Weil-Deligne representations of $W_K$ over domains containing $\scrO$ and the traces of these lifts are parametrized by a pseudorepresentation $T:W_K \to \scrO$. 

The eigenvarieties are an important source of examples of families of Galois representations. The traces of the Galois representations attached to the arithmetic points of an eigenvariety are interpolated by a pseudorepresentation defined over the global sections of the eigenvariety. Unfortunately, this pseudorepresentation does not lift to a Galois representation over the global sections. However, by \cite[Lemma 7.8.11]{BellaicheChenevierAsterisQUE}, around each nonempty admissible open affinoid subset $U$, the pseudorepresentation defined over the global sections of an eigenvariety lifts to a Galois representation on a finite type module over some integral extension of the normalization of $\calO(U)$. But this module is not known to be free over its coefficient ring. So theorem \ref{Thm: conductor pseudorepresentation} cannot be applied to eigenvarieties to study the tame conductors of all arithmetic points. To circumvent this problem, we establish theorem \ref{Thm: conductor sum} which can be used to study the tame conductors of a large class of arithmetic points. Theorem \ref{Thm: conductor big}, \ref{Thm: conductor pseudorepresentation}, \ref{Thm: conductor sum} apply to $p$-adic families of automorphic representations, for instance, to ordinary families, overconvergent families, and explain the variation of the tame conductors. We illustrate it using the example of Hida family of ordinary automorphic representations for definite unitary groups (see theorem \ref{Thm: Application: Hida unitary}). 

\section{Preliminaries}
\label{Sec: Preliminaries}
For every field $F$, we fix an algebraic closure $\overline F$ of it and denote by $G_F$ the Galois group $\Gal(\overline F/F)$. For a finite place $v$ of a number field $E$, the decomposition group $\Gal(\overline E_v/E_v)$ is denoted by $G_v$. Let $W_v\subset G_v$ (resp. $I_v\subset G_v$) denote the Weil group (resp. inertia group) and $\Fr_v\in G_v/I_v$ denote the geometric Frobenius element. The fraction field of a domain $A$ is denoted by $Q(A)$ and the field $\overline{Q(A)}$ is denoted by $\overline Q (A)$. If $\frakp$ is a prime ideal of a ring $A$, then the mod $\frakp$ reduction map is denoted by $\pi_\frakp$. The integral closure of a domain $R$ in $\Qbar (R)$ is denoted by $R^\intal$. For each map $f:R\to S$ between domains, we fix an extension $f^\intal :R^\intal \to S^\intal$ of $f$. 

Let $q$ denote the cardinality of the residue field $k$ of the ring of integers $\calO_K$ of $K$. Let $I_K$ denote the inertia subgroup of $G_K$. Let $\{G_K^s\}_{s\geq -1}$ denote the upper numbering filtration on $G_K$ by ramification subgroups. Fix an element $\phi\in G_K$ which lifts the geometric Frobenius $\Fr_k\in G_k$. The Weil group $W_K$ is defined as the subgroup of $G_K$ consisting of elements which map to an integral power of the geometric Frobenius element $\Fr_k$ in $G_k$. Its topology is determined by decreeing that $I_K$ is open and has its subspace topology induced from $G_K$. Define $v_K:W_K\to \bbZ$ by $\sigma|_{K^\ur}=\Fr_k^{v_K(\sigma)}$ for all $\sigma\in W_K$. 

\begin{definition}[{\cite[8.4.1]{DeligneConstantesDesEquationsFunctional}}]
\label{Defn: Weil-Deligne representations}
Let $A$ be a commutative domain of characteristic zero. A \textnormal{Weil-Deligne representation} of $W_K$ on a free $A$-module $M$ of finite rank is a triple $(r,M,N)$ consisting of a representation $r:W_K\to \Aut_A (M)$ with open kernel and a nilpotent endomorphism $N\in \End_A(M)$ such that 
$$r(\sigma)N r(\sigma)^{-1}=q^{-v_K(\sigma)}N$$ 
for all $\sigma\in W_K$. The operator $N$ is called the \textnormal{monodromy} of $(r, M, N)$. 
\end{definition}

Suppose $(r,V,N)$ is a Weil-Deligne representation with coefficients in a field $L$ of characteristic zero which contains the characteristic roots of all elements of $r(W_K)$. Let $r(\phi)=r(\phi)^{ss}u=ur(\phi)^{ss}$ be the Jordan decomposition of $r(\phi)$ as the product of a diagonalizable matrix $r(\phi)^{ss}$ and a unipotent matrix $u$. Following \cite[8.5]{DeligneConstantesDesEquationsFunctional}, define $\tilde r(\sigma)=r(\sigma) u^{-v_K(\sigma)}$ for all $\sigma\in W_K$. Then $(\tilde r, V, N)$ is a Weil-Deligne representation (by \cite[8.5]{DeligneConstantesDesEquationsFunctional}) and is called the {\it Frobenius-semisimplification} of $(r,V, N)$ (\cf\,\cite[8.6]{DeligneConstantesDesEquationsFunctional}). It is denoted by $V^\Frss$. 

\begin{definition}
If $(r, V, N)$ is Weil-Deligne representation of $W_K$ on a vector space $V$ over an algebraically closed field of characteristic zero, then its \textnormal{conductor} is defined as 
$$\cond (r,V, N) = V^{I_K, N=0} + \int_0 ^\infty \codim V^{G_K^u} du $$
where $V^{I_K, N=0}$ denotes the subspace of $V$ on which $I_K$ acts trivially and $N$ is zero. 
\end{definition}

\begin{lemma}
\label{Lemma: invariant space}
Let $f:A\to B$ be a map between domains of characteristic zero. Let $E$ (resp. $F$) be a field containing $A$ (resp. $B$). Let $\rho:G\to \gln_n(A)$ be a representation of a group $G$. Then the $E$-dimension of the space of $G$-invariants of the representation $G \xra{\rho} \gln_n(A) \hra \gln_n(E)$ and the $F$-dimension of the space of $G$-invariants of the representation $G \xra{\rho} \gln_n(A) \xra{f} \gln_n(B) \hra \gln_n(F)$ are equal if $G$ is finite. 
\end{lemma}

\begin{proof}
If $r: G \to \gln_n(\calF)$ is a representation where $\calF$ is a field of characteristic zero, then the dimension of the space of $G$-invariants of $r$ is equal to the trace of the idempotent operator $\frac 1{|G|} \sum_{g\in G} r(g)$. Since $f(n)$ is equal to $n$ for any integer $n$, the lemma follows. 
\end{proof}

\section{Main results}
\label{Sec: Main results}
Denote the fraction field of $\scrO$ by $\scrL$ and the algebraic closure of $\bbQ$ in $\scrO$ by $\bbQ^\cl$. 

\begin{theorem}
\label{Thm: conductor big}
Let $(r, N): W_K \to \gln_n(\scrO)$ be a Weil-Deligne representation of $W_K$ with coefficients in $\scrO$. If $f:\scrO \to \bbQpbar$ is a map such that $f\circ (r, N)$ is pure, then the conductor of $f\circ (r, N)$ is equal to the conductor of $(r, N)\otimes _\scrO \Qbar(\scrO)$.
\end{theorem}

\begin{proof}
By \cite[Theorem 4.1]{BigPuritySub}, there exist positive integers $m, t_1, \cdots, t_m$, unramified characters $\chi_1, \cdots, \chi_m:W_K \to (\scrO^\intal)^\times$, representations $\rho_1, \cdots, \rho_m$ of $W_K$ over $\bbQ^\cl$ with finite image such that 
$$((r, N)\otimes _\scrO \Qbar(\scrO)) ^\Frss \simeq \bigoplus_{i=1}^n \Sp_{t_i} (\chi_i \otimes \rho_i)_{/\scrLbar}, $$
$$
(f\circ (r, N) )^\Frss\simeq \bigoplus_{i=1}^n \Sp_{t_i} (f^\intal \circ (\chi_i \otimes \rho_i))_{/\bbQpbar}$$
So the dimension of $((r, N)\otimes _\scrO \Qbar(\scrO))^{I_K, N=0}$ (resp. $(f\circ (r, N))^{I_K, N=0}$) over $\scrLbar$ (resp. $\bbQpbar$) is equal to $\sum_{i=1}^n \rho_i^{I_K}$ (resp. $\sum_{i=1}^n (f^\intal \circ \rho_i)^{I_K}$). By lemma \ref{Lemma: invariant space}, the $\scrLbar$-dimension of $((r, N)\otimes _\scrO \Qbar(\scrO))^{I_K, N=0}$ is equal to the $\bbQpbar$-dimension of $(f\circ (r, N))^{I_K, N=0}$. Note that $G_K^u$ is contained in $I_K$ for any $u\geq 0$ and $r(H)$ is finite for any subgroup $H$ of $I_K$. Hence $\dim_\scrLbar ((r, N)\otimes _\scrO \Qbar(\scrO))^H$ is equal to $\dim _\bbQpbar (f \circ (r, N))^H$ by lemma \ref{Lemma: invariant space}. This shows that the conductor of $f\circ (r, N)$ is equal to the conductor of $(r, N)\otimes _\scrO \Qbar(\scrO)$.
\end{proof}

Now we establish an analogue of the above result for pseudorepresentations of Weil groups. The notion of pseudorepresentations was introduced by Wiles \cite{WilesOrdinaryLambdaAdic} in the two-dimensional case, and by Taylor \cite{TaylorGaloisReprAssociatedToSiegelModForms} in full generality. They are defined by abstracting the crucial properties of the trace of a group representation. 

\begin{theorem}
\label{Thm: conductor pseudorepresentation}
Let $\calO$ be an integral domain and $\res:\scrO \hra\calO$ be an injective map. Let $T: W_K \to \scrO$ be a pseudorepresentation of dimension $n\geq 1$ and let $(r, N):W_K \to \gln_n(\calO)$ be a Weil-Deligne representation such that $\res\circ T = \tr r$ and $f \circ (r, N)$ is pure for some map $f: \calO \to \bbQpbar$. Then there exists a non-negative integer $C$ such that the following statement holds. If $(r', N'): W_K\to \gln_n(\calO')$ is a Weil-Deligne representation over a domain $\calO'$ satisfying
\begin{enumerate}
\item $\res'\circ T = \tr r'$ for some injective map $\res':\scrO \hra \calO'$, 
\item $f'\circ (r', N')$ is pure for some map $f': \calO' \to \bbQpbar$, 
\end{enumerate}
then 
$$\cond \left( (r', N') \otimes_{\calO'} \Qbar (\calO') \right) = \cond \left(f'\circ (r', N') \right) = C.$$
\end{theorem}
\begin{proof}
Note that by \cite[Theorem 5.4, Proposition 2.8]{BigPuritySub}, there exist positive integers $m, t_1, \cdots, t_m$, unramified characters $\chi_1, \cdots, \chi_m:W_K \to (\scrO^\intal)^\times$, representations $\rho_1, \cdots, \rho_m$ of $W_K$ over $\bbQ^\cl$ with finite image such that 
$$((r', N')\otimes_{\calO'}\overline Q(\calO') )^\Frss \simeq \bigoplus_{i=1}^m \Sp_{t_i} (\res'^\intal \circ (\chi_i \otimes \rho_i))$$
for any Weil-Deligne representation $(r', N'): W_K\to \gln_n(\calO')$ over a domain $\scrO'$ satisfying condition (1), (2). Then lemma \ref{Lemma: invariant space} shows that the conductors of $(r', N') \otimes_{\calO'} \Qbar (\calO')$ and $\oplus_{i=1}^m \Sp_{t_i} (\chi_i \otimes \rho_i)$ are equal. So by theorem \ref{Thm: conductor big}, the result follows by taking $C$ to be the conductor of $\oplus_{i=1}^m \Sp_{t_i} (\chi_i \otimes \rho_i)$. 
\end{proof}

Let $F$ be a number field and $T:G_F \to \scrO$ be a pseudorepresentation such that $T=T_1+\cdots+T_n$ where $T_1:G_F \to \scrO, \cdots, T_n:G_F\to \scrO$ are pseudorepresentations. Using \cite[Theorem 1]{TaylorGaloisReprAssociatedToSiegelModForms}, choose semisimple representations $\sigma_1, \cdots, \sigma_n$ of $G_F$ over $\scrLbar$ such that $\tr \sigma_i = T_i$ for all $1\leq i \leq n$. Fix a finite place $w$ of $F$ not dividing $p$ and assume that $\scrO$ is a $\bbZp$-algebra.

\begin{definition}
\label{Defn: locus}
The \textnormal{irreducibility and purity locus} of $T_1, \cdots, T_n$ is defined to be the collection of all tuples of the form $(\calO,  \frakm, \kappa,  \loc,\rho_1$, $\cdots, \rho_n)$ where $\calO$ is a $\bbZp$-algebra and it is a Henselian Hausdorff domain with maximal ideal $\frakm$, $\kappa$ denotes the residue field of $\calO$ and is an algebraic extension of $\bbQp$, $\loc:\scrO \hra\calO$ is an injective $\bbZp$-algebra homomorphism and for each $1\leq i\leq n$, $\rho_i$ is an irreducible $G_F$-representation over $\overline \kappa$ such that the trace of $\rho_i$ is equal to $\loc \circ T_i\modu \frakm$ and $\rho_i|_{G_w}$ is pure. 
\end{definition}

\begin{theorem}
\label{Thm: conductor sum}
Suppose the irreducibility and purity locus of $T_1, \cdots, T_n$ is nonempty and the restrictions of $\sigma_1, \cdots, \sigma_n$ to $W_w$ are monodromic. Then for any element $(\calO, \frakm, \kappa, \loc, \rho_1, \cdots, \rho_n)$ in the irreducibility and purity locus of $T_1$, $\cdots$, $T_n$, we have 
$$\cond \left(\WD\left(\bigoplus_{i=1}^n \rho_i |_{W_w} \right)\right) 
= \cond \left(\WD\left(\bigoplus_{i=1}^n \sigma_i |_{W_w} \right)\right).$$
\end{theorem}
\begin{proof}
By \cite[Theorem 5.6, Proposition 2.8]{BigPuritySub}, there exist positive integers $m, t_1, \cdots, t_m$, unramified characters $\chi_1, \cdots, \chi_m:W_K \to (\scrO^\intal)^\times$, representations $\theta_1, \cdots, \theta_m$ of $W_K$ over $\bbQ^\cl$ with finite image such that $\WD\left(\oplus_{i=1}^n \rho_i |_{W_w} \right)^\Frss$ is isomorphic to $\oplus_{i=1}^m \Sp_{t_i} (\pi_\frakm^\intal \circ \loc^\intal \circ (\chi_i \otimes \theta_i))$ and $\WD \left(\oplus_{i=1}^n \sigma_i |_{W_w} \right)^\Frss$ is isomorphic to $\oplus_{i=1}^m \Sp_{t_i} (\chi_i \otimes \theta_i)$. Then lemma \ref{Lemma: invariant space} shows that the conductors of $\WD\left(\oplus_{i=1}^n \rho_i |_{W_w} \right)$ and $\WD \left(\oplus_{i=1}^n \sigma_i |_{W_w} \right)$ are equal to the conductor of $\oplus_{i=1}^m \Sp_{t_i} (\chi_i \otimes \theta_i)$. The result follows.  
\end{proof}

\section{Conductors in families}
In this section, using the example of Hida family of ordinary automorphic representations for definite unitary groups, we show how results of section \ref{Sec: Main results} describes the variation of tame conductors in $p$-adic families. 

Let $F$ be a CM field, $F^+$ be its maximal totally real subfield. Let $n\geq 2$ be an integer and assume that if $n$ is even, then $n[F^+:\bbQ]$ is divisible by 4. Let $\ell>n$ be a rational prime and assume that every prime of $F^+$ lying above $\ell$  splits in $F$. Let $K$ be a finite extension of $\bbQl$ in $\bbQlbar$ which contains the image of every embedding $F\hra \bbQlbar$. Let $S_\ell$ denote the set of places of $F^+$ above $\ell$. Let $R$ denote a finite set of finite places of $F^+$ disjoint from $S_\ell$ and consisting of places which split in $F$. For each place $v\in S_\ell \cup R$, choose once and for all a place $\widetilde v$ of $F$ lying above $v$. For $v\in R$, let $\Iw(\widetilde v)$ be the compact open subgroup of $\gln_n(\calO_{F_{\widetilde v}})$ and $\chi_v$ be the character as in \cite[\S 2.1, 2.2]{GeraghtyPhDThesis}. 

Let $G$ be the reductive algebraic group over $F^+$ as in \cite[\S 2.1]{GeraghtyPhDThesis}. For each dominant weight $\lambda$ (as in \cite[Definition 2.2.3]{GeraghtyPhDThesis}) for $G$, the group $G(\bbA_{F^+}^{\infty, R})\times \prod_{v\in R} \Iw(\widetilde v)$ acts on the spaces $S_{\lambda, \{\chi_v\}}(\bbQlbar)$, $S_{\lambda, \{\chi_v\}}^\ord(\calO_K)$ (as in \cite[Definition 2.2.4, 2.4.2]{GeraghtyPhDThesis}). For an irreducible constituent $\pi$ of the $G(\bbA_{F^+}^{\infty, R})\times \prod_{v\in R} \Iw(\widetilde v)$-representation $S_{\lambda, \{\chi_v\}}(\bbQlbar)$, let $\WBC(\pi)$ denote the weak base change of $\pi$ to $\gln_n(\bbA_F)$ (which exists by \cite[Corollaire 5.3]{LabesseChangementDeBaseCMDiscreteSeries}) and let $r_\pi:G_F\to \gln_n(\bbQlbar)$ (as in \cite[Proposition 2.7.2]{GeraghtyPhDThesis}) denote the unique (up to equivalence) continuous semisimple representation attached to $\WBC(\pi)$ via \cite[Theorem 3.2.3]{ChenevierHarrisConstructionAutoGalRepr}. 

An irreducible constituent $\pi$ of the $G(\bbA_{F^+}^{\infty, R})\times \prod_{v\in R} \Iw(\widetilde v)$-representation $S_{\lambda, \{\chi_v\}}(\bbQlbar)$ is said to be an {\it ordinary automorphic representation for $G$} if $\pi^{U(\frakl^{b,c})}\cap S_{\lambda, \{\chi_v\}}^\ord(U(\frakl^{b,c}), \calO_K)\neq 0$ for some integers $0\leq b\leq c$ (see \cite[Definition 2.2.4, \S 2.3]{GeraghtyPhDThesis} for details). Let $U$ be a compact open subgroup of $G(\bbA_{F^+}^\infty)$, $T$ be a finite set of finite places of $F^+$ containing $R\cup S_\ell$ and such that every place in $T$ splits in $F$ (see \cite[\S 2.3]{GeraghtyPhDThesis}). Let $\bbT^\ord$ denote the universal ordinary Hecke algebra $\bbT^{T, \ord}_{\{\chi_v\}}(U(\frakl^\infty), \calO_K)$ (as in \cite[Definition 2.6.2]{GeraghtyPhDThesis}). Let $\Lambda$ be the completed group algebra as in \cite[Definition 2.5.1]{GeraghtyPhDThesis}. By definition of $\bbT^\ord$, it is equipped with a $\Lambda$-algebra structure and is finite over $\Lambda$. An $\calO_K$-algebra homomorphism $f:A \to \bbQlbar$ is said to be an {\it arithmetic specialization} of a finite $\Lambda$-algebra $A$ if $\ker (f|_\Lambda)$ is equal to the prime ideal $\wp_{\lambda, \alpha}$ (as in \cite[Definition 2.6.3]{GeraghtyPhDThesis}) of $\Lambda$ for some dominant weight $\lambda$ for $G$ and a finite order character $\alpha:T_n(\frakl) \to \calO_K^\times$. By \cite[Lemma 2.6.4]{GeraghtyPhDThesis}, each arithmetic specialization $\eta$ of $\bbT^\ord$ determines an ordinary automorphic representation $\pi_\eta$ for $G$. An arithmetic specialization $\eta$ of $\bbT^\ord$ is said to be {\it stable} if $\WBC(\pi_\eta)$ is cuspidal. 

Let $\frakm$ be a non-Eisenstein maximal ideal of $\bbT^\ord$ (in the sense of \cite[\S 2.7]{GeraghtyPhDThesis}). Let $r_\frakm$ denote the representation of $G_{F^+}$ as in \cite[Proposition 2.7.4]{GeraghtyPhDThesis}. Then by restricting it to $G_F$ and then composing with the projection $\gln_n(\bbT^\ord_\frakm)\times \gln_1(\bbT^\ord_\frakm)\to \gln_n(\bbT^\ord_\frakm)$, we get a continuous representation $G_F \to \gln_n(\bbT^\ord_\frakm)$ which is denoted by $r_\frakm$ by abuse of notation. Since $\frakm$ is non-Eisenstein, the $G_F$-representations $\eta\circ r_\frakm$ and $r_{\pi_\eta}$ are isomorphic for any arithmetic specialization $\eta$ of $\bbT^\ord_\frakm$ (by \cite[Proposition 2.7.2, 2.7.4]{GeraghtyPhDThesis}). 

\begin{theorem}
\label{Thm: Application: Hida unitary}
Let $w\nmid \ell$ be a finite place of $F$, $\fraka$ be a minimal prime of $\bbT^\ord_\frakm$ and $\frakm$ be the maximal ideal of $\bbT^\ord$ containing $\fraka$. Suppose $\frakm$ is non-Eisenstein. Denote the quotient ring $\bbT^\ord_\frakm/\fraka$ by $\calR(\fraka)$ and the representation $r_\frakm \modu \fraka$ by $r_\fraka$. Then there exists a non-negative integer $C$ such that the conductor of $\WD(r_{\pi_\eta}|_{W_w})$ is equal to $C$ for any stable arithmetic specialization $\eta$ of $\calR(\fraka)$. Moreover, the conductor of $\WBC(\pi_\eta)_w$ is also equal to $C$ for any such specialization $\eta$. 
\end{theorem}
\begin{proof}
If $\pi$ is an irreducible constituent of the $G(\bbA_{F^+}^{\infty, R})\times \prod_{v\in R} \Iw(\widetilde v)$-representation $S_{\lambda, \{\chi_v\}}(\bbQlbar)$ such that $\WBC(\pi)$ is cuspidal, then for any finite place $w$ of $F$ not dividing $\ell$, $r_\pi|_{G_w}$ is pure by \cite[Theorem 1.1, 1.2]{CaraianiLocalGlobalCompatibility} and the proofs of theorem 5.8, corollary 5.9 of \loccit. Note that $r_\fraka|_{W_w}$ is monodromic by Grothendieck's monodromy theorem (see \cite[p.\,515--516]{SerreTate}). So theorem \ref{Thm: conductor big} gives the first part. Since local Langlands correspondence preserves conductors, the rest follows from \cite[Theorem 1.1]{CaraianiLocalGlobalCompatibility} on local-global compatibility of cuspidal automorphic representations for $\gln_n$. 
\end{proof}

\subsection*{Acknowledgements}
It is a pleasure to thank Olivier Fouquet for his guidance during the preparation of this article. I also thank the Max Planck Institute for Mathematics for providing an optimal working condition.

\providecommand{\bysame}{\leavevmode\hbox to3em{\hrulefill}\thinspace}
\providecommand{\MR}{\relax\ifhmode\unskip\space\fi MR }
\providecommand{\MRhref}[2]{%
  \href{http://www.ams.org/mathscinet-getitem?mr=#1}{#2}
}
\providecommand{\href}[2]{#2}

\end{document}